\documentclass[10pt]{amsart}
\usepackage{amsmath}
\usepackage{amscd}
\usepackage{amssymb}
\usepackage{amsthm}
\RequirePackage{filecontents}
\usepackage{fullpage}
\usepackage{longtable}
\usepackage[numbers]{natbib}  
\usepackage[draft]{hyperref}

\usepackage{tikz}
\usetikzlibrary{arrows}
\usetikzlibrary{positioning}
\usepackage[
top = 2.5cm,
bottom = 2.5cm,
left = 3cm,
right = 3cm]{geometry}

\theoremstyle{plain}
\newtheorem{prop}{Proposition}
\newtheorem{lem}[prop]{Lemma}

\theoremstyle{definition}

\newtheorem{rem}[prop]{Remark}
\newtheorem{ex}[prop]{Example}

\newenvironment{psmallmatrix}
  {\left(\begin{smallmatrix}}
  {\end{smallmatrix}\right)}

\title{The rings of Hilbert modular forms for $\mathbb{Q}(\sqrt{29})$ and $\mathbb{Q}(\sqrt{37})$}
\author{Brandon Williams}

\subjclass[2010]{11F27,11F41}
\address{Fachbereich Mathematik \\ Technische Universit\"at Darmstadt \\ 64289 Darmstadt, Germany}

\email{bwilliams@mathematik.tu-darmstadt.de}

\begin{document}

\begin{abstract} We use Borcherds products and their restrictions to Hirzebruch-Zagier curves to determine generators and relations for the graded rings of Hilbert modular forms for the fields $\mathbb{Q}(\sqrt{29})$ and $\mathbb{Q}(\sqrt{37})$. These seem to be the first cases where the graded ring can be computed despite obstructions to the existence of Borcherds products with arbitrary divisors.
\end{abstract}

\maketitle

\section{Introduction}

The problem considered in this note is to determine generators and relations for the graded rings $M_*(\Gamma_K)$ of Hilbert modular forms over real-quadratic fields $K$. There is some history to this. The structure has been worked out in the cases $\mathbb{Q}(\sqrt{5})$ (Gundlach, \cite{Gu}), $\mathbb{Q}(\sqrt{2})$ (Hammond, \cite{H}), $\mathbb{Q}(\sqrt{13})$ (van der Geer and Zagier, \cite{GZ}), $\mathbb{Q}(\sqrt{17})$ (Hermann, \cite{He} determined the structure of the subring of symmetric forms; the full ring appears in section 6.5 of \cite{M}), and $\mathbb{Q}(\sqrt{6})$ (van der Geer, \cite{G}). Section 6.2 of \cite{M} gives a more detailed overview. Simpler derivations for the fields $\mathbb{Q}(\sqrt{5}), \mathbb{Q}(\sqrt{13}), \mathbb{Q}(\sqrt{17})$ were given by Mayer \cite{M} using a reduction process against Borcherds products with simple zeros on the diagonal, and the work here extends this idea. \\

We only consider fields $K = \mathbb{Q}(\sqrt{p})$ of prime discriminant $p \equiv 1 \, (4)$ and class number one as this makes a number of aspects of the theory simpler (but there are ultimately workarounds for both of these assumptions). The first few such primes are $p=5,13,17,$ and the corresponding graded rings have the form

\begin{align*} M_*(\Gamma_{\mathbb{Q}(\sqrt{5})}) &= \mathbb{C}[X_2,X_5,X_6,X_{15}]/(R_{30}), \\ M_*(\Gamma_{\mathbb{Q}(\sqrt{13})}) &= \mathbb{C}[X_2,X_3,X_4,X_5,X_6,X_8,X_9] / (R_9,R_{10},R_{12},R_{18}), \\ M_*(\Gamma_{\mathbb{Q}(\sqrt{17})}) &= \mathbb{C}[X_2,X_3,X_4^{(1)},X_4^{(2)},X_5^{(1)},X_5^{(2)},X_6^{(1)},X_6^{(2)},X_7,X_8,X_9] / \mathcal{I}, \\ &\quad\quad\quad \mathcal{I} = (R_9,R_{10},R_{11},R_{12}^{(1)},R_{12}^{(2)},R_{13},R_{14},R_{18}),\end{align*} where $X_k,X_k^{(i)}$ denote modular forms of weight $k$ and where $R_k,R_k^{(i)}$ denote relations that are homogeneous of weight $k$, see \cite{M}. \\

No results seem to have been written down for larger primes. We could expect this to be more difficult for primes $p >17$, $p \equiv 1\, (4)$ because the Weil representation associated to the finite quadratic module $(\mathcal{O}_K^{\#}/ \mathcal{O}_K, N_{K/\mathbb{Q}})$ for the lattice of integers $\mathcal{O}_K$ admits nonzero cusp forms of weight $2$. More precisely, there are $\lceil \frac{p-5}{24} \rceil$ linearly independent cusp forms of weight $2$ as is well-known (and originally due to Hecke). For one thing, these cusp forms lift injectively under the Doi-Naganuma map to holomorphic differentials on the Hilbert modular surfaces such that those surfaces are never rational (see also \cite{HV}, chapter 3 for the classification of Hilbert modular surfaces) so one roughly expects the rings in question to be more complicated. More immediately, these cusp forms act as obstructions to the possible divisors of Borcherds products in the sense of section 3.2 of \cite{H1} so it is generally not possible to find modular forms with simple divisors to reduce against. \\

In this note we are able to solve this problem for the smallest two cases $\mathbb{Q}(\sqrt{29})$ and $\mathbb{Q}(\sqrt{37})$ where obstructions occur by considering the restrictions of Hilbert modular forms both to the diagonal and to other Hirzebruch-Zagier curves and by matching these simultaneously to polynomials in Eisenstein series and Borcherds products. The graded rings have the form \begin{align*} M_*(\Gamma_{\mathbb{Q}(\sqrt{29})}) &= \mathbb{C}[X_2^{(1)},X_2^{(2)},X_3^{(1)},X_3^{(2)},X_4,X_5,X_6^{(1)},X_6^{(2)},X_7,X_8,X_9] / \mathcal{I}_{29}, \\ M_*(\Gamma_{\mathbb{Q}(\sqrt{37})}) &= \mathbb{C}[X_2^{(1)},X_2^{(2)},X_3^{(1)},X_3^{(2)},X_3^{(3)},X_4^{(1)},X_4^{(2)},X_5^{(1)},X_5^{(2)},X_6^{(1)},X_6^{(2)},X_6^{(3)},X_7,X_8,X_9] / \mathcal{I}_{37}, \end{align*} where both $\mathcal{I}_{29}$ and $\mathcal{I}_{37}$ are homogeneous ideals of relations of weight between $6$ and $18$. We found minimal systems of $35$ generators of $\mathcal{I}_{29}$ and $77$ generators of $\mathcal{I}_{37}$. These are given explicitly in the ancillary material on arXiv and the author's website. \\

The methods used here likely apply to some other real-quadratic fields. Unfortunately the construction of generators is rather ad hoc, and we are unable to answer some natural questions. For example, are the graded rings always generated by Eisenstein series and Borcherds products? (Some preliminary calculations with the field $\mathbb{Q}(\sqrt{41})$ suggest that they may not be enough to produce all modular forms of weight 5.) Also, if $K \ne \mathbb{Q}(\sqrt{5})$, then are the graded rings always presented by generators of weight at most 9 and by relations of weight at most $18$? (The last question is similar to Main Theorem 1.4 of \cite{VZ}, which implies a statement of this type for all modular curves but which does not seem to apply to the Hilbert modular surfaces in an obvious way.) \\

Computing the Fourier expansion of a Borcherds product is not entirely trivial to do, even in the relatively simple case of Hilbert modular forms. (The paper \cite{GKR} discusses that problem in general and offers some improvements over the naive method of computing.) The ideals of relations would probably be difficult to work out directly. In this note we are able to avoid that problem entirely by using Koecher's principle and the existence of certain Borcherds products with convenient divisors. In particular we never need to compute the Fourier expansion of any Hilbert modular form explicitly. We only use the interpretation of restrictions along various curves on the Hilbert modular surface as elliptic modular forms, which are easier to compute. \\

\textbf{Acknowledgments.} A large part of what is presented below is the result of computations in SAGE and Macaulay2. Some easier computations were carried out on a CoCalc server. The Magma implementation of Hilbert modular forms as described in \cite{DV} was useful for checking correctness. I thank Martin Raum for providing resources with which the more time-consuming computations were performed. I am also grateful to Jan Hendrik Bruinier and John Voight for helpful discussions. This work was supported by the LOEWE research unit Uniformized Structures in Arithmetic and Geometry.

\section{Hilbert modular forms}

\textbf{2.1 Hilbert modular forms.} We quickly recall some of the basic theory of Hilbert modular forms attached to quadratic fields. The book \cite{H1} introduces Hilbert modular forms with a view toward Borcherds products and is therefore a useful reference for more details. Let $\mathbb{H}$ be the usual upper half-plane and let $K$ be a real-quadratic field of discriminant $d_K$ with ring of integers $\mathcal{O}_K$. For simplicity we assume that $K$ has class number one. A \textbf{Hilbert modular form} of (parallel) weight $k \in \mathbb{N}_0$ is a holomorphic function of two variables $f : \mathbb{H} \times \mathbb{H} \rightarrow \mathbb{C}$ which satisfies $$f\Big( \frac{a \tau_1 + b}{c \tau_1 + d}, \frac{a' \tau_2 + b'}{c' \tau_2 + d'} \Big) = (c \tau_1 + d)^k (c' \tau_2 + d')^k f(\tau_1,\tau_2)$$ for all $M = \begin{psmallmatrix} a & b \\ c & d \end{psmallmatrix} \in \Gamma_K = SL_2(\mathcal{O}_K)$. Here $a'$ denotes the conjugate of $a \in K$. \\

The invariance of $f$ under all translations $T_b = \begin{psmallmatrix} 1 & b \\ 0 & 1 \end{psmallmatrix}$, $b \in \mathcal{O}_K$ implies that $f$ has a Fourier series which we write in the form $$f(\tau_1, \tau_2) = \sum_{\nu \in \mathcal{O}_K^{\#}} a(\nu) q_1^{\nu} q_2^{\nu'}, \; \; q_i = e^{2\pi i \tau_i}, \; a(\nu) \in \mathbb{C}.$$ Here $\mathcal{O}_K^{\#}$ is the codifferent $$\mathcal{O}_K^{\#} = \{\nu \in K: \; \mathrm{Tr}_{K/\mathbb{Q}}(\nu x) \in \mathbb{Z} \; \text{for all} \; x \in \mathcal{O}_K\} = \sqrt{1 / d_K} \mathcal{O}_K.$$ The G\"otzky-Koecher principle (\cite{H1}, theorem 1.20) states that $a(\nu) = 0$ whenever either $\nu < 0$ or $\nu' < 0$ so the usual boundedness condition at cusps for automorphic forms is automatically satisfied. \\

Some of the simplest Hilbert modular forms are the \textbf{Hecke Eisenstein series} $$\mathbf{E}_k(\tau_1,\tau_2) = 1 + \frac{4}{\zeta_K(1-k)} \sum_{\substack{\nu \in \mathcal{O}_K^{\#} \\ \nu,\nu' > 0}} \sigma_{k-1}\Big( \nu \sqrt{d_K} \mathcal{O}_K \Big) q_1^{\nu} q_2^{\nu'}, \; \; k \ge 2 \; \text{even},$$ where $\zeta_K(s)$ is the Dedekind zeta function of $K$ and where $\sigma_{k-1}(\mathfrak{a}) = \sum_{\mathfrak{c} | \mathfrak{a}} N_{K/\mathbb{Q}}(\mathfrak{c})^{k-1}$ is the ideal divisor sum. When $k \ge 4$ these arise out of an averaging process similar to the Eisenstein series for $SL_2(\mathbb{Z})$, and the case $k=2$ can be recovered using Hecke's convergence trick. See \cite{H1}, section 1.5 for details. \\

Cusp forms are Hilbert modular forms that also vanish at the cusps: $(\infty,\infty)$ and $(a,a')$ for all $a \in K$. Up to equivalence by $\Gamma_K$ the number of cusps is exactly the class number of $K$. We have assumed that $K$ has class number one and therefore a form $f(\tau_1,\tau_2) = \sum_{\nu} c(\nu) q_1^{\nu} q_2^{\nu'}$ is a cusp form if and only if $c(0) = 0$. \\

One can generalize Hilbert modular forms to include a character of $SL_2(\mathcal{O}_K)$. The character group is always finite and depends up to isomorphism only on the ramification behavior of $2$ and $3$ in $\mathcal{O}_K$ (see \cite{BS}). We will only consider \textbf{symmetric characters} $\chi$: those for which $\chi(M) = \chi(M')$ for all $M \in SL_2(\mathcal{O}_K)$. This is equivalent to $f(\tau_1,\tau_2)$ and $f^{\sigma}(\tau_1,\tau_2) = f(\tau_2,\tau_1)$ having the same character. Any Hilbert modular form $f$ for a symmetric character can be decomposed into its symmetric and antisymmetric parts as $$f = f^+ + f^-, \;\; f^+ = \frac{f + f^{\sigma}}{2}, \; f^- = \frac{f - f^{\sigma}}{2}.$$ (Here $f$ is symmetric or antisymmetric if $f = f^{\sigma}$ or $f = -f^{\sigma}$, respectively.) \\

\textbf{2.2 Graded rings of Hilbert modular forms.} Let $K$ be a real-quadratic number field and set $\Gamma_K = SL_2(\mathcal{O}_K)$. The set $M_{*,char}(\Gamma_K)$ of all Hilbert modular forms with all characters is naturally a graded ring where forms are graded by both their weights and characters. We denote by $M_{*,sym}(\Gamma_K)$ and $M_*(\Gamma_K)$ the subrings of Hilbert modular forms with symmetric and trivial characters, respectively. \\

The graded rings $M_*(\Gamma_K)$ are always finitely generated. In fact the Baily-Borel theorem implies that their projective spectra $\mathrm{Proj} \, M_{*}(\Gamma_K)$ are complex surfaces. Together with the Noether normalization lemma this implies that we can always find a set of three algebraically independent forms $f_1,f_2,f_3$ such that $M_*(\Gamma_K)$ is finite over $\mathbb{C}[f_1,f_2,f_3]$ (i.e. finitely generated as a module). If $f_1,f_2,f_3$ have weights $k_1,k_2,k_3$ respectively then the theory of Hilbert series implies that there is a polynomial $p(t)$ for which $$\sum_{k=0}^{\infty} \mathrm{dim} \, M_k(\Gamma_K) t^k = \frac{p(t)}{(1 - t^{k_1})(1 - t^{k_2})(1- t^{k_3})}.$$ Of course neither $\{f_1,f_2,f_3\}$ nor their weights $\{k_1,k_2,k_3\}$ are unique. Since $M_{*,char}(\Gamma_K)$ is finite over $M_*(\Gamma_K)$ (as some power of any Hilbert modular form will have trivial character), we get the same result for the larger ring. In particular, for every character $\chi$ of $\Gamma_K$ there is a polynomial $p_{\chi}(t)$ such that, with the same exponents $k_1,k_2,k_3$, $$\sum_{k=0}^{\infty} \mathrm{dim} \, M_{k,\chi}(\Gamma_K) t^k = \frac{p_{\chi}(t)}{(1 - t^{k_1})(1 - t^{k_2})(1- t^{k_3})}.$$

\begin{rem} It was proved in \cite{TV} (theorem 3.4) that the Hilbert series of even-weight modular forms for every real quadratic field except $\mathbb{Q}(\sqrt{5})$ has the form $$\sum_{k=0}^{\infty} \mathrm{dim}\, M_{2k}(\Gamma_K) t^k = \frac{p(t)}{(1 - t)^2 (1 - t^3)}$$ for an explicit polynomial $p(t)$ of degree exactly $6$. This is not a special case of the paragraph above, and it does not hold for the full ring of modular forms.
\end{rem}

\textbf{2.3 Restrictions to Hirzebruch-Zagier curves.} Hilbert modular forms can be restricted to the Hirzebruch-Zagier cycles of \cite{HZ} to produce elliptic modular forms for $\Gamma_0(n)$. Recall that the Hirzebruch-Zagier cycle of discriminant $n$ is the set $T_n \subseteq \mathbb{H} \times \mathbb{H}$ of all points $(\tau_1,\tau_2)$ that satisfy an equation of the form $$a \tau_1 \tau_2 + \lambda \tau_1 + \lambda' \tau_2 + b = 0$$ for some $a,b \in \mathbb{Z}$ and $\lambda \in \mathcal{O}_K^{\#}$ with $ab - \lambda \lambda' = n$. These cycles are $\Gamma_K$-invariant and there are inclusions $T_n \subseteq T_{nd^2}$ for all $n,d \in \mathbb{N}$ which are strict if $d > 1$. \\

Suppose $\lambda \in \mathcal{O}_K$ is a totally positive integer of norm $\ell = \lambda \lambda'$. Then the function $$g(\tau) = \mathrm{Res}_{\lambda} f(\tau) = f(\lambda \tau, \lambda' \tau)$$ satisfies \begin{align*} g(M \cdot \tau) &= f\Big( \lambda \frac{a \tau + b}{\ell c \tau + d}, \lambda' \frac{a \tau + b}{\ell c \tau + d} \Big) \\ &= f\left( \begin{psmallmatrix} a & \lambda b \\ \lambda' c & d \end{psmallmatrix} \cdot (\lambda \tau), \begin{psmallmatrix} a & \lambda' b \\ \lambda c & d \end{psmallmatrix} \cdot (\lambda' \tau) \right) \\ &= \chi \left( \begin{psmallmatrix} a & \lambda b \\ \lambda' c & d \end{psmallmatrix} \right) (\ell c \tau + d)^{2k} g(\tau) \end{align*} for all $M = \begin{psmallmatrix} a & b \\ \ell c & d \end{psmallmatrix} \in \Gamma_0(\ell).$ Moreover the growth condition at cusps follows from that of $f$ and therefore $g(\tau)$ is a modular form of weight $2k$ and level $\Gamma_0(\ell)$ for the character $$\chi_g \left( \begin{psmallmatrix} a & b \\ \ell c & d \end{psmallmatrix} \right) = \chi \left( \begin{psmallmatrix} a & \lambda b \\ \lambda' c & d \end{psmallmatrix} \right).$$ In addition, if $\ell$ is prime and $f$ is symmetric, then we find \begin{align*} g(-1/\ell \tau) &= f\Big(-\frac{\lambda}{\ell \tau}, -\frac{\lambda'}{\ell \tau} \Big) \\ &= f\Big( -\frac{1}{\lambda' \tau}, -\frac{1}{\lambda \tau}\Big) \\ &= f\Big(-\frac{1}{\lambda \tau}, -\frac{1}{\lambda' \tau} \Big) \\ &= (\lambda \tau)^k (\lambda' \tau)^k f(\lambda \tau, \lambda' \tau) \\ &= \ell^k \tau^{2k} g(\tau), \end{align*} i.e. $g$ is an eigenform of the Atkin-Lehner involution $W_{\ell} g(\tau) = \ell^{-k} \tau^{-2k} g(-1/\ell \tau)$ with eigenvalue $+1$. Similarly, if $f$ is antisymmetric, then $g$ will have eigenvalue $-1$ under $W_{\ell}$. \\

\begin{rem} If $\ell$ is a prime and $\lambda$ is a fixed, totally positive integer of norm $\ell$, then the Hirzebruch-Zagier cycle $T_{\ell}$ is exactly the orbit of $$\{(\lambda \tau, \lambda' \tau): \; \tau \in \mathbb{H}\} \cup \{(\lambda' \tau, \lambda \tau): \; \tau \in \mathbb{H}\}$$ under $SL_2(\mathcal{O}_K)$. In that sense $f(\lambda \tau, \lambda' \tau)$ is a restriction of $f$ to $T_{\ell}$. In particular, if $f$ is either symmetric or antisymmetric (which includes Borcherds products), then it vanishes along the entire divisor $T_{\ell}$ whenever $f(\lambda \tau, \lambda' \tau) \equiv 0$ for any such element $\lambda$. Also note that $T_1$ is simply the orbit of the diagonal.
\end{rem}

%\begin{rem} Since we assume that $K$ has prime discriminant, it follows that the fundamental unit $\varepsilon$ of $K$ has norm $-1$.  Therefore, the choice of $\lambda$ almost does not matter: if $p = \lambda \lambda' = \mu \mu'$ are two splittings then we can write $\mu \in \{\varepsilon^{2n} \lambda, \varepsilon^{2n} \lambda'\}$ for some $n \in \mathbb{Z}$, and therefore either $$f(\mu \tau, \mu' \tau) = f(M \cdot (\lambda \tau), M' \cdot (\lambda' \tau)) = f(\lambda \tau, \lambda' \tau)$$ or $f(\mu \tau, \mu' \tau) = f(\lambda' \tau, \lambda \tau)$, where $M$ is the matrix $\begin{pmatrix} \varepsilon^n & 0 \\ 0 & \varepsilon^{-n} \end{pmatrix} \in SL_2(\mathcal{O}_K)$.  $f(\lambda \tau, \lambda' \tau)$ and $f(\lambda' \tau, \lambda \tau)$ can be considered the restrictions of $f$ to $T_p$ in the sense that the curve $T_p$ is exactly the orbit of $\{(\lambda \tau, \lambda' \tau): \; \tau \in \mathbb{H}\} \cup \{(\lambda' \tau, \lambda \tau): \; \tau \in \mathbb{H}\}$ under $SL_2(\mathcal{O}_K)$.
%\end{rem}

\begin{rem} Even if the Hilbert modular form $f(\tau_1,\tau_2)$ vanishes along the diagonal, we can take its Taylor expansion in the variable $\tau_1$ about the point $\tau_2$: $$f(\tau_1,\tau_2) = g(\tau_2) (\tau_1 - \tau_2)^N + O((\tau_1 - \tau_2)^{N+1}), \; \; |\tau_1 - \tau_2| < y_2,$$ for some function $g$. For any matrix $M = \begin{psmallmatrix} a & b \\ c & d \end{psmallmatrix} \in SL_2(\mathbb{Z})$ compare $$f(M \cdot \tau_1, M \cdot \tau_2) = (c \tau_1 + d)^{-N} (c \tau_2 + d)^{-N} g(M \cdot \tau_2) (\tau_1 - \tau_2)^N + O((\tau_1 - \tau_2)^{N+1})$$ with $$(c \tau_1 + d)^k (c \tau_2 + d)^k f(M \cdot \tau_1, M \cdot \tau_2) = (c \tau_1 + d)^k (c \tau_2 + d)^k g(\tau_2) (\tau_1 - \tau_2)^N + O((\tau_1 - \tau_2)^{N+1})$$ to see that $g$ is a modular form of weight $2k+2N$. If $N \ge 1$, then $(\tau_1 - \tau_2)^N$ is not bounded, and therefore the growth condition on $f$ forces $g$ to be a cusp form. Similar expansions can be written out for Hilbert modular forms that vanish along the other Hirzebruch-Zagier curves, but we will not need any of these except for the diagonal.
\end{rem}

\textbf{2.4 Forced zeros.} Modular forms for $SL_2(\mathbb{Z})$ of certain weights have forced zeros at elliptic CM points. Namely, if the weight of a level one modular form $f$ is not divisible by $4$, then $f(i) = 0$; and if the weight is not divisible by $3$, then $f(\rho) = 0$ where $\rho =e^{2\pi i / 3}$. Some statements for Hilbert modular forms that are roughly in this spirit will be important later. We assume that $K = \mathbb{Q}(\sqrt{p})$ has prime discriminant $p$ and fundamental unit $\varepsilon$.

\begin{lem} (i) Every antisymmetric Hilbert modular form vanishes along the diagonal.  \\ (ii) Every symmetric Hilbert modular form of weight $k$ and character $\chi$ for which $\chi(\begin{psmallmatrix} \varepsilon & 0 \\ 0 & \varepsilon^{-1} \end{psmallmatrix}) \ne (-1)^k$ vanishes on the Hirzebruch-Zagier divisor $T_p$ of discriminant $p$. \\ (iii) Every antisymmetric Hilbert modular form of weight $k$ and character $\chi$ with $\chi(\begin{psmallmatrix} \varepsilon & 0 \\ 0 & \varepsilon^{-1} \end{psmallmatrix}) \ne (-1)^{k+1}$ vanishes on the Hirzebruch-Zagier divisor $T_p$ of discriminant $p$.
\end{lem}
\begin{proof} (i) $f(\tau_1,\tau_2) = -f(\tau_2,\tau_1)$ implies $f(\tau,\tau) = 0$ identically. \\ (ii) Since $K$ has prime discriminant, the fundamental unit $\varepsilon$ has negative norm. In particular $\lambda = \varepsilon \sqrt{p}$ is totally positive and $T_p$ is the $\Gamma_K$-orbit of $\{(\lambda \tau, \lambda'\tau) : \; \tau \in \mathbb{H}\}$. We find $$f(\lambda \tau, \lambda' \tau) = \chi \left( \begin{psmallmatrix} \varepsilon & 0 \\ 0 & \varepsilon^{-1} \end{psmallmatrix} \right) (-1)^k f(\lambda' \tau, \lambda \tau) = \chi\left( \begin{psmallmatrix} \varepsilon & 0 \\ 0 & \varepsilon^{-1} \end{psmallmatrix} \right) (-1)^k f(\lambda \tau, \lambda' \tau),$$ using the transformation under $\begin{psmallmatrix} \varepsilon & 0 \\ 0 & \varepsilon^{-1} \end{psmallmatrix}$ which sends $(\lambda' \tau, \lambda \tau)$ to $(\lambda \tau, \lambda' \tau)$ and using the symmetry of $f$. This implies the claim. \\ (iii) This uses nearly the same argument as (ii).
\end{proof}

\section{Two kinds of Borcherds products}

\textbf{3.1 Borcherds products for $\Gamma_0(p)$.} Let $p$ be a prime. We recall a generalization of theorem 14.1 of \cite{B1} which produces modular products of level $\Gamma_0(p)$. (This is itself a special case of theorem 13.3 of \cite{B2}.) Let $A_p$ denote the Kohnen plus space of nearly-holomorphic modular forms of weight $1/2$ and level $(\Gamma_0(4p),\chi_{\vartheta})$ with rational Fourier coefficients for the multiplier system $\chi_{\vartheta}$ of the classical theta function $\vartheta(\tau) = 1 + 2q + 2q^4 + 2q^9 + ...$ In other words $A_p$ consists of Laurent series $$f(q) = \sum_{n \gg -\infty} c(n) q^n, \; q = e^{2\pi i \tau}, \; c(n) \in \mathbb{Q}$$ for which $c(n) = 0$ if $n$ is a quadratic nonresidue mod $4p$, and for which $f \cdot \vartheta$ is a modular form of weight $1$ and level $\Gamma_0(4p)$ for the Dirichlet character $$\chi(n) = \begin{cases} \left( \frac{-1}{n}\right) : & n \; \text{coprime to} \; 4p \\ 0: & \text{otherwise}. \end{cases}$$

Define $$\delta_p(n) = \begin{cases} 1: & n \equiv 0 \, (p) \\ 1/2: & n \not \equiv 0 \, (p), \end{cases}$$ and for $f(\tau) = \sum_n c(n) q^n$ let $\tilde c(n) = c(n) \delta_p(n)$. For any modular form $f \in A_p$ as above, Borcherds' theorem states that there is a rational number $h \in \mathbb{Q}$ (the ``Weyl vector") such that $$\Psi_f(\tau) = q^h \cdot \prod_{n=1}^{\infty} (1 - q^n)^{\tilde c(n^2)}$$ converges for large enough $\mathrm{im}(\tau)$ and extends to a multivalued modular form of weight $c(0)$ and level $\Gamma_0(p)$ and some character. All zeros and singularities $w$ of $\Psi_f$ on $\mathbb{H}$ satisfy a quadratic equation of the form $aw^2 + bw + c = 0$ with $a,b,c \in \mathbb{Z}$ and $a \equiv 0 \, (p)$, and the order of $\Psi_f$ at such a point is a sum of coefficients for negative exponents in $f$: $$\mathrm{ord}(\Psi_f;w) = \sum_{m=1}^{\infty} \tilde c(m^2(b^2 - 4ac)).$$ In particular, $\Psi_f$ has no branch cuts if these orders are all integral, and it is a holomorphic away from the cusps if these orders are also nonnegative. The Weyl vector $h$ can be calculated as the constant term of the product of $f$ with a mock Eisenstein series of weight $3/2$ and level $4p$ which is closely related to Zagier's Eisenstein series \cite{Z}. Of course one can also expand the product above and simply guess the (unique) value of $h$.

\begin{ex} Let $p = 3$; then $A_3$ contains the input function $$f(\tau) = q^{-3} + 1 - 12q + 42q^4 - 76q^9 + 168q^{12} - 378q^{13} + 690q^{16} - 897q^{21} + 1456q^{24} - 3468q^{25} \pm ...$$ which can be computed in SAGE to any desired precision by multiplying the weight four form $$1 -  2q + q^3 - 12q^4 + 24q^5 + 44q^7 - 108q^8 - 2q^9 + 84q^{11} - 84q^{12} + O(q^{13}) \in M_4(\Gamma_0(12))$$ by the eta product $\frac{\eta(2\tau)}{\eta(\tau)^2 \eta(12\tau)^6}$. Its lift $\Psi_f$ is a modular form of weight $1$ and level $(\Gamma_0(3),\chi)$ which has simple zeros exactly on the $\Gamma_0(3)$-orbit of the CM point $w = \frac{1}{2} + \frac{\sqrt{3}}{6}i$ (i.e. the solution $w \in \mathbb{H}$ of $3w^2 - 3w + 1 = 0$). The Weyl vector is $h=0$. From the Fourier expansion $$\Psi_f(\tau) = (1 - q)^{-12/2} (1 - q^2)^{42/2} (1-q^3)^{-76} (1 - q^4)^{690/2} (1 - q^5)^{-3468/2}...= 1 + 6q + 6q^3 + 6q^4 + O(q^6)$$ we see that $\Psi_f$ is the theta function of $x^2 + xy + y^2$ and that $\chi$ is the nontrivial Nebentypus mod $3$.
\end{ex}

\begin{rem} Borcherds products are more naturally orthogonal modular forms; in the special case above, one should really consider $\Gamma_0(p)$ acting by conjugation on a certain lattice of symmetric matrices and preserving the determinant (a quadratic form of signature $(1,2)$). The Atkin-Lehner involution $W_p f(\tau) = p^{-k/2}\tau^{-k} f(-\frac{1}{p\tau})$ also preserves the determinant, and therefore all Borcherds products $\Psi_f$ are eigenforms of $W_p$. We should also mention that passing to vector-valued input functions makes the condition that $p$ is prime or even squarefree unnecessary. \\
\end{rem}

\textbf{3.2 Borcherds products for $\mathbb{Q}(\sqrt{p})$.} We recall the construction of Hilbert forms by Borcherds products (using \cite{B2},\cite{BB}; see also section 3.2 of \cite{H1} for an introduction). Let $p \equiv 1 \, (4)$ be a prime discriminant and let $K = \mathbb{Q}(\sqrt{p})$ with ring of integers $\mathcal{O}_K = \mathbb{Z}[\frac{1 + \sqrt{p}}{2}]$. One can essentially identify Hilbert modular forms of level $\mathcal{O}_K$ with orthogonal modular forms for the signature $(2,2)$ lattice $\mathcal{O}_K \oplus II_{1,1}$, where $\mathcal{O}_K$ carries the norm-form $N_{K/\mathbb{Q}}$. \\

Let $\chi$ denote the quadratic Dirichlet character attached to $K$: $\chi(n) = \left(\frac{n}{p}\right)$, and let $M_0^!(\Gamma_0(p),\chi)$ denote the space of nearly-holomorphic modular forms of weight zero, level $\Gamma_0(p)$ and Nebentypus $\chi$. Denote by $B_p$ the plus-space $$B_p = \Big\{ F = \sum_{n \gg -\infty} c(n) q^n \in M_0^!(\Gamma_0(p),\chi): \; c(n) \in \mathbb{Q}, \; c(n) = 0 \; \text{if} \; \chi(n) = -1\Big\}.$$ As before we define $\delta_p(n) = 1$ if $n \equiv 0 \, (p)$ and $\delta_p(n) = 1/2$ otherwise, and set $\tilde c(n) = c(n) \delta_p(n)$ for any series $F = \sum_n c(n) q^n$. \\

Let $f \in B_p$. As in section 3.2 of \cite{H1}, the principal part of $F$ determines a splitting of $\mathbb{R}_{>0} \times \mathbb{R}_{>0}$ into \textbf{Weyl chambers} $W$. These are the connected components that are left after removing the curves $\{(\lambda y, \lambda' y): \, y \in \mathbb{R}_{>0}\}$ for all totally positive $\lambda \in \mathcal{O}_K$ with $c(-\lambda \lambda') \ne 0$. For any Weyl chamber $W$, Borcherds' theorem states that there is a number $h_W \in K$ (the ``Weyl vector") such that $$\Psi_F(\tau_1,\tau_2) = q_1^{h_W} q_2^{h_W'} \prod_{\substack{\nu \in \mathcal{O}_K^{\#} \\ \langle \nu, W \rangle > 0}} (1 - q_1^{\nu} q_2^{\nu'})^{\tilde c(p \cdot N_{K/\mathbb{Q}}\nu)}$$ converges for all $(\tau_1,\tau_2) \in \mathbb{H} \times \mathbb{H}$ with $\mathrm{im}(\tau_1\tau_2)$ sufficiently large and that $\Psi_F$ extends to a multivalued Hilbert modular form of weight $c(0)/2$ for some character. Here $\langle \nu, W \rangle > 0$ means that $\mathrm{Tr}_{K/\mathbb{Q}}(\nu w) > 0$ for all $w \in K$ with $(w,w') \in W$. The divisor of $\Psi_F$ is a sum of Hirzebruch-Zagier cycles: $$\mathrm{div} \, \Psi_F = \sum_{n=1}^{\infty} \tilde c(-n) T_n.$$ In particular, $\Psi_F$ is holomorphic if and only if all orders $\sum_{m=1}^{\infty} \tilde c(-nm^2)$ are nonnegative integral. (Note that this does not imply that all $c(-n)$ are nonnegative.)

\begin{rem} The Borcherds products attached to the various Weyl chambers and a single input function all have the same weight and divisor so the quotient between any two is a nonzero constant (by the G\"otzky-Koecher principle). That constant is not generally $1$. From this it follows that the Borcherds products attached to the different Weyl chambers all have the same character as well.
\end{rem}

\begin{rem} A holomorphic Borcherds product $\Psi_F$ is a cusp form if and only if the compactification of its divisor intersects all cusps. This too can be read off the input function. When $K$ has class number one, it is sufficient for the input function $f$ to have a nonzero coefficient in any exponent $q^{-n}$ with $n \in N_{K/\mathbb{Q}} \mathcal{O}_K$. Therefore Hilbert modular form Borcherds products of small weight tend to be cusp forms. Indeed all Borcherds products we construct in this note for $\mathbb{Q}(\sqrt{29})$ and $\mathbb{Q}(\sqrt{37})$ are cusp forms. \\
\end{rem}

\textbf{3.3 Restricting Borcherds products.} Let $K = \mathbb{Q}(\sqrt{p})$ for a prime $p \equiv 1 \, (4)$. Suppose $\Psi_F(\tau_1,\tau_2) = q_1^h q_2^{h'} \prod_{\langle \nu, W \rangle > 0} (1 - q_1^{\nu} q_2^{\nu'})^{\tilde c(p \cdot N_{K/\mathbb{Q}} \nu)}$ is the expansion of the Borcherds lift of $$F(\tau) = \sum_{n \in \mathbb{Z}} c(n) q^n$$ within some Weyl chamber $W$, and set $\tilde c(n) = c(n)$ or $c(n)/2$ as in section 3.2. Also suppose that $\lambda \in \mathcal{O}_K$ satisfies $(\lambda,\lambda') \in W$ (in particular, $\lambda$ is totally positive) and that $N_{K/\mathbb{Q}}(\lambda) = \ell$ is $1$ or a prime. Restricting gives

$$\Psi_F(\lambda \tau, \lambda' \tau) = q^{\mathrm{Tr}(\lambda h)} \prod_{n=1}^{\infty} (1 - q^n)^{\tilde b(n^2)}$$ for the coefficients $$\tilde b(n^2) = \sum_{\substack{ \langle \nu, W \rangle > 0 \\ \mathrm{Tr}(\lambda \nu) = n}}  \tilde c(p \cdot N_{K/\mathbb{Q}} \nu) = \sum_{\mathrm{Tr}(\lambda \nu) = n} \tilde c(p \cdot N_{K/\mathbb{Q}} \nu),$$ where we use the fact that $\mathrm{Tr}(\lambda \nu) > 0$ implies $\langle \nu, W \rangle > 0$, by lemma 3.2 of \cite{Br}. This suggests that $\mathrm{Res}_{\lambda} \Psi_F(\tau) = \Psi_F(\lambda \tau, \lambda' \tau)$ itself is likely a Borcherds product as in section 2.1. This is true and for $\ell \ne p$ the input function can be produced from $F$ as follows. Divide the coefficients of all $q^n$, $n \not \equiv 0 \, (p)$ in $F$ by two (to get the modified coefficients $\tilde c(n)$), change variables $q \mapsto q^{4\ell/p}$, multiply the result by $\vartheta(\tau/p)$ where $\vartheta(\tau) = 1 + 2q + 2q^4 + 2q^9 + ...$ is the theta function, and restrict to integer exponents. In other words we take the input function $$f(\tau) = \sum_{n=0}^{\infty} \sum_{\substack{r \in \mathbb{Z} \\ 4\ell n + r^2 \equiv 0 \, (p)}} \tilde c(n) q^{(4\ell n + r^2)/p} = \sum_{m=0}^{\infty} b(m) q^m, \; \; b(m) = \sum_{\substack{r \in \mathbb{Z} \\ mp - r^2 \equiv 0 \, (4\ell)}} \tilde c\left( \frac{mp-r^2}{4\ell}\right).$$

The procedure for constructing $f$ (and in particular the proof that $f \in A_{\ell}$ is a valid input function in the sense of section 3.1) is the same as remark 10 and example 11 of \cite{W} with minor changes. That is, the vector-valued modular form corresponding to $F$ is transformed into a weak Jacobi form $\varphi(\tau,z)$ of fractional index $\frac{p}{4\ell}$ through the theta decomposition and then $f(\tau)$ is the Kohnen plus form corresponding to the vector-valued modular form $\varphi(\tau,0)$. This is ultimately the same procedure as the theta-contraction of \cite{Ma}.\\

The square coefficients are $$b(n^2) = \sum_{\substack{r \in \mathbb{Z} \\ pn^2 - r^2 \equiv 0 \, (4\ell)}} \tilde c\left( p \cdot N_{K/\mathbb{Q}}\nu\right), \; \; \nu = \frac{n - r / \sqrt{p}}{2 \lambda}.$$ Note that if $n \not \equiv 0 \, (\ell)$ and $pn^2 - r^2 \equiv 0 \, (4\ell)$, then exactly one of $\sqrt{p}n \pm r$ is divisible by $2\lambda$ in $\mathcal{O}_K$ (and the other is divisible by $2\lambda'$) because $\lambda$ is prime and $\lambda \nmid \mathrm{gcd}(\sqrt{p}n+r,\sqrt{p}n-r)$. (Remember that we assume $\mathcal{O}_K$ has class number one.) If $n \equiv 0 \, (\ell)$ and therefore also $r \equiv 0 \, (\ell)$, then both $\sqrt{p}n \pm r$ are divisible by $2\ell$ and therefore by $2\lambda$. Altogether it follows that the $\tilde b(n^2) = b(n^2) \delta_p(n)$ satisfy

$$\tilde b(n^2) = \sum_{\substack{\nu \in \mathcal{O}_K^{\#} \\ \mathrm{Tr}_{K/\mathbb{Q}}(\lambda \nu) = n}} \tilde c(p \cdot N_{K/\mathbb{Q}}\nu) = \begin{cases} b(n^2)/2: \; n \not \equiv 0 \, (p) \\ b(n^2): \; n \equiv 0 \, (p), \end{cases}$$ and therefore that $\mathrm{Res}_{\lambda} \Psi_F(\tau) = \Psi_f(\tau)$ holds. (That the Weyl vectors match up follows from the fact that $\mathrm{Res}_{\lambda} \Psi_F(\tau)$ and $\Psi_f(\tau)$ are meromorphic modular forms of the same weight and level whose Fourier expansions are equal up to a shift; so their quotient is a power of $q = e^{2\pi i \tau}$ which is invariant under $SL_2(\mathbb{Z})$ and therefore equals $1$. It does not seem to be as easy to show that the Weyl vectors match directly from their definitions.) \\

When $\ell = p$ the procedure of theta contraction is similar. The only difference in the result is that we do not divide any of the coefficients of $F$ by 2; instead we multiply $F(4\tau) \vartheta(\tau/p)$ directly and restrict to integer exponents.

\begin{ex} Let $p = 5$. Then $B_5$ contains the input function $$F(\tau) = 2q^{-1} + 10 + 22q - 108q^4 + 110q^5 + 88q^6 - 790q^9 + 680q^{10} \pm ...$$ which lifts to Gundlach's weight 5 cusp form $s_5(\tau_1,\tau_2)$ for $\mathbb{Q}(\sqrt{5})$ that vanishes exactly on the $\Gamma_K$-orbit of the diagonal \cite{Gu}. There is a Weyl chamber of $F$ with Weyl vector $h = \frac{1}{2} - \frac{\sqrt{5}}{10}$. We compute the restriction to the Hirzebruch-Zagier curve $T_5$ as follows. Take the totally positive element $\lambda_5 = \frac{5 + \sqrt{5}}{2}$ of norm $5$, which satisfies $\mathrm{Tr}_{K/\mathbb{Q}}(\lambda_5 h) > 0$. We multiply \begin{align*} &\quad F(4\tau) \vartheta(\tau/5) \\ &=\Big( 2q^{-4} + 10 + 22q^4 - 108q^{16} + 110q^{20} + 88q^{24} \pm ... \Big) \Big( 1 + 2q^{1/5} + 2q^{4/5} + 2q^{9/5} + ... \Big) \\ &= 2q^{-4} + 4q^{-19/5} + 4q^{-16/5} + 4q^{-11/5} + 4q^{-4/5} + 10 + 20q^{1/5} + 20q^{4/5} + 4q + ...\end{align*} and restrict to integer exponents to find the input function $$f(\tau) = 2q^{-4} + 10 + 4q + 22q^4 + 20q^5 + 44q^9 - 104q^{16} + 130q^{20} \pm ... \in A_5.$$ Its lift is \begin{align*} s_5(\lambda_5 \tau, \lambda_5' \tau) &= q^{\mathrm{Tr}_{K/\mathbb{Q}}(\lambda_5 h)} (1 - q)^2(1-q^2)^{11} (1 - q^3)^{22} (1 - q^4)^{-52}... \\ &= q^2 (1 - 2q - 10q^2 + 140q^4 + ...) \\ &= \frac{1}{4} \eta(\tau)^8 \eta(5\tau)^8 (5 E_2(5\tau) - E_2(\tau)), \end{align*} where $E_2(\tau) = 1 - 24 \sum_{n=1}^{\infty} \sigma_1(n) q^n$ with $\sigma_1(n) = \sum_{d | n} d$, and $\eta(\tau) = q^{1/24} \prod_{n=1}^{\infty} (1 - q^n)$. In particular, the case of Borcherds' theorem in section 3.1 shows that this has simple zeros exactly on the $\Gamma_0(5)$-orbit of the discriminant $4$ CM point $\tau = \frac{2+i}{5}$, so $(\lambda_5 \tau,\lambda_5' \tau)$ represents the intersection $T_1 \cap T_5$. \\

To compute the restriction onto the curve $T_{11}$ with $\lambda_{11} = \frac{7 + \sqrt{5}}{2}$, we have to pass from the coefficients $c(n)$ to $\tilde c(n)$. Therefore we multiply \begin{align*} &\Big( q^{-44/5} + 10 + 11q^{44/5} - 54q^{176/5} + 110q^{44} + 44q^{264/5} \pm ... \Big) \Big(1 + 2q^{1/5} + 2q^{4/5} + 2q^{9/5} + ... \Big) \\ &= q^{-44/5} + 2q^{-43/5} + 2q^{-8} + 2q^{-7} + 2q^{-28/5} + 2q^{-19/5} + 2q^{-8/5} + 10 + 20q^{1/5} + 20q^{4/5} + 2q + ... \end{align*} and restrict to integer exponents to find the input form $$f(\tau) = 2q^{-8} + 2q^{-7} + 10 + 2q + 2q^4 + 20q^5 + 22q^9 + 22q^{12} + ... \in A_{11}$$ whose Borcherds lift is \begin{align*} s_5(\lambda_{11} \tau, \lambda_{11}' \tau) &= q^{\mathrm{Tr}_{K/\mathbb{Q}}(\lambda_{11} h)} (1 - q)(1-q^2)(1-q^3)^{11} ... \\ &= q^3 (1 - q - q^2 - 10q^3 + 10q^5 + 121q^6 \pm ...) \in S_{10}(\Gamma_0(11)). \end{align*} Note that $\lambda_5$ and $\lambda_{11}$ lie in the same Weyl chamber of $F$ so there really is a single Hilbert modular form with these restrictions. In general this would only hold up to a constant multiple as in remark 7.
\end{ex}

\section{The graded ring of Hilbert modular forms for $\mathbb{Q}(\sqrt{29})$}

Let $K = \mathbb{Q}(\sqrt{29})$ with ring of integers $\mathcal{O}_K = \mathbb{Z}[\frac{1 + \sqrt{29}}{2}]$. Since $2$ and $3$ both remain inert in $\mathcal{O}_K$, there are no nontrivial characters of $SL_2(\mathcal{O}_K)$. \\

The principal parts that extend to nearly-holomorphic forms in $M_0^!(\Gamma_0(29),\chi)$ can be determined by theorem 3.49 of \cite{H1}; essentially the only obstruction is that the product with the unique cusp form in the plus space $$q - 3q^4 - 3q^5 + 5q^6 + 2q^7 - 2q^9 - q^{13} - q^{16} \pm ... \in S_2(\Gamma_0(29),\chi)$$ has constant term zero. In this way one can prove that there exist nearly-holomorphic modular forms $F_2,F_3,G_3,F_4,F_6,G_6$ of weight $0$ and level $\Gamma_0(29)$ and Nebentypus $\chi(d) = \left(\frac{29}{d}\right)$ whose Fourier expansions begin as follows: \begin{align*} F_2(\tau) &= 2q^{-4} + 6q^{-1} + 4 + 2q - 2q^4 + 8q^5 - 2q^6 + 14q^7 + 4q^9 - 10q^{13} \pm ... \\ F_3(\tau) &= 2q^{-6} + 4q^{-4} + 2q^{-1} + 6 - 6q - 2q^4 + 12q^5 + 16q^6 + 18q^7 + 28q^9 - 10q^{13}  \pm ... \\ G_3(\tau) &= 2q^{-5} + 6q^{-1} + 6 + 10q + 4q^4 + 12q^5 + 0q^6 + 0q^7 - 6q^9 + 20q^{13} \pm ... \\ F_4(\tau) &= 2q^{-6} + 2q^{-5} + 2q^{-4} + 2q^{-1} + 8 + 2q + 4q^4 + 16q^5 + 18q^6 + 4q^7 + 18q^9 + 20q^{13} \pm ... \\ F_6(\tau) &= q^{-29} + 6q^{-1} + 12 + 30q + 124q^4 + 162q^5 + 252q^6 + 336q^7 + 648q^9 + 2050q^{13} + ... \\  G_6(\tau) &= 2q^{-16} + 2q^{-6} + 2q^{-5} - 2q^{-1} + 12 -10q - 34q^4 - 48q^5 - 70q^6 + 80q^7 + 186q^9 + 336q^{13} \pm ... \end{align*}

%F_3(\tau) &= 2q^{-9} + 4q^{-1} + 6 - 8q + 16q^4 - 2q^5 + 20q^6 + 34q^7 - 38q^9 + 38q^{13} \pm ...
%
%F_4(\tau) &= 2q^{-7} + 2q^{-4} + 2q^{-1} + 8 + 12q + 14q^4 - 2q^5 - 2q^6 + 14q^7 + 30q^9 - 28q^{13} \pm ...  
%F_5(\tau) &= 2q^{-7} + 2q^{-5} + 2q^{-1} + 10 + 20q + 20q^4 + 2q^5 + 2q^6 + 20q^7 + 20q^9 + 2q^{13} \pm ...
%G_6(\tau) &= 2q^{-38} + 12 - 88q + 484q^4 - 312q^5 + 750q^6 + 1408q^7 - 3696q^9 + 6952q^{13}
%F_7(\tau) &= q^{-29} + 2q^{-6} + 2q^{-4} + 2q^{-1} + 14 + 22q + 124q^4 + 166q^5 + 270q^6 + 340q^7 + 672q^9 + 2050q^{13} + ...

The Borcherds lifts $\phi_i = \Psi_{F_i}$, $\psi_i = \Psi_{G_i}$ have weight $i$, and their divisors can be read off the principal parts of the input functions:

$$\mathrm{div}\, \phi_2 = 3T_1 + T_4, \; \; \mathrm{div} \, \phi_3 = T_1 + 2T_4 + T_6, \; \; \mathrm{div} \, \psi_3 = 3T_1 + T_5,$$ $$\mathrm{div} \, \phi_4 = T_1 + T_4 + T_5 + T_6, \; \; \mathrm{div} \, \phi_6 = 3T_1 + T_{29}, \; \; \mathrm{div} \, \psi_6 = -T_1 + T_5 + T_6 + T_{16}.$$

We will also need the following (holomorphic) quotients of the above forms as generators:

\begin{align*} \phi_5 &= \frac{\psi_3 \phi_4}{\phi_2}, \; \; \mathrm{div} \, \phi_5 = T_1 + 2T_5 + T_6; \\ \phi_7 &= \frac{\phi_3 \phi_6}{\phi_2}, \; \; \mathrm{div} \, \phi_7 = T_1 + T_4 + T_6 + T_{29}; \\ \phi_8 &= \frac{\phi_3 \psi_3 \phi_6}{\phi_2^2}, \; \; \mathrm{div} \, \phi_8 = T_1 + T_5 + T_6 + T_{29}; \\ \phi_9 &= \frac{\phi_3^2 \psi_3 \phi_6}{\phi_2^3}, \; \; \mathrm{div} \, \phi_9 = -T_1 + T_4 + T_5 + 2T_6 + T_{29}.\end{align*}

All are cusp forms. We fix $\lambda_1 = 1$ and $\lambda_5 = \frac{7 + \sqrt{29}}{2}$ and compute restrictions up to a constant multiple (which does not matter here and is omitted by abuse of notation) using the procedure in section 3.3:

\begin{align*} \mathrm{Res}_{\lambda_1} \psi_6(\tau) &= \Delta(\tau) = q - 24q^2 + 252q^3 - 1472q^4 + 4830q^5 \pm ... \\ \mathrm{Res}_{\lambda_1} \phi_9(\tau) &= \Delta(\tau) E_6(\tau) = q - 528q^2 - 4284q^3 + 147712q^4 \pm ... \\ \mathrm{Res}_{\lambda_1} \mathbf{E}_2(\tau) &= E_4(\tau) = 1 + 240q + 2160q^2 + 6720q^3 + 17520q^4 + ...\end{align*} (none of which actually require computation because the respective spaces of modular forms or cusp forms are one-dimensional) and

\begin{align*} \mathrm{Res}_{\lambda_5} \phi_2(\tau) &= s_4(\tau) = q - 4q^2 + 2q^3 + 8q^4 - 5q^5 \pm ... \\ \mathrm{Res}_{\lambda_5} \phi_3(\tau) &= s_4(\tau) e_2(\tau) = q + 2q^2 - 4q^3 - 28q^4 + 25q^5 \pm ... \\ \mathrm{Res}_{\lambda_5} \phi_4(\tau) &= s_4(\tau)^2 = q^2 - 8q^3 + 20q^4 - 70q^6 \pm ... \\ \mathrm{Res}_{\lambda_5} \phi_6(\tau) &= s_4(\tau)^2 e_4(\tau) = q^2 - 18q^3 + 10q^4 + 240q^5 \pm ... \\ \mathrm{Res}_{\lambda_5} \phi_7(\tau) &= e_2(\tau) e_4(\tau) s_4(\tau)^2 = q^2 - 12q^3 - 80q^4 + 870q^6 \pm ... \\ \mathrm{Res}_{\lambda_5} \mathbf{E}_2(\tau) &= e_2(\tau)^2 - 4s_4(\tau) = 1 + 8q +  88q^2 + 256q^3 + 664q^4 + 1400q^5 + ...\end{align*} with the omitted restrictions above being zero, where we have fixed the following generators of the ring of modular forms for $\Gamma_0(5)$:
% \mathrm{Res}_{\lambda_5} \psi_6(\tau) &= (e_2(\tau)^2 - 4s_4(\tau))(e_2(\tau)^4 + 12e_2(\tau)^2 s_4(\tau) + 144s_4(\tau)^2) = 1 + 44q + 904q^2 + 9064q^3 + 83560q^4 + ... \\

 \begin{align*} e_2(\tau) &= \frac{5 E_2(5 \tau) - E_2(\tau)}{4} = 1 + 6q + 18q^2 + 24q^3 + 42q^4 + 6q^5 + ... \\ e_4(\tau) &= \frac{25 E_4(5 \tau) - E_5(\tau)}{24} = 1 - 10q - 90q^2 - 280q^3 - 730q^4 - 1010q^5 - ... \\ s_4(\tau) &= \eta(\tau)^4 \eta(5\tau)^4 = q - 4q^2 + 2q^3 + 8q^4 - 5q^5 \pm ... \end{align*} which satisfy a single relation $$e_4^2 = e_2^4 - 44 s_4 e_2^2 - 16s_4^2$$ in weight $8$. Note that $e_2,e_4$ are eigenforms of the Atkin-Lehner involution $W_5$ with eigenvalue $-1$ while $s_4$ has eigenvalue $+1$ under $W_5$.

\begin{lem} Let $f$ be a Hilbert modular form of parallel weight for $\mathbb{Q}(\sqrt{29})$ and let $\lambda_5 = \frac{7 + \sqrt{29}}{2}$. Then the restriction $f(\lambda_5 \tau, \lambda_5' \tau)$ of $f$ to $T_5$ coincides with the restriction of some polynomial in $\mathbf{E}_2,\phi_2,\phi_3,\phi_6,\phi_7$.
\end{lem}

\begin{proof} Any form $f$ can be split into its symmetric and antisymmetric parts, and it is enough to argue for each of those parts separately. We also argue differently depending on whether the weight is even or odd. \\

 Case 1: suppose $f \in M_k^{sym}(\Gamma_K)$ is symmetric and has even weight. Then $\mathrm{Res}_{\lambda_5} f$ is an eigenform of $W_5$ with eigenvalue $+1$ and has weight divisible by $4$ and is therefore a polynomial expression $P$ in $e_2^2,s_4$. Since $$e_2^2 = \mathrm{Res}_{\lambda_5}(\mathbf{E}_2 + 4\phi_2), \; \; s_4 = \mathrm{Res}_{\lambda_5} \phi_2$$ we find $$\mathrm{Res}_{\lambda_5} f = \mathrm{Res}_{\lambda_5} P(\mathbf{E}_2 + 4\phi_2, \phi_2) = \mathrm{Res}_{\lambda_5} \tilde P(\mathbf{E}_2,\phi_2)$$ for some polynomial $\tilde P$. \\

Case 2: suppose $f \in M_k^{anti}(\Gamma_K)$ is antisymmetric and has odd weight. In particular $f$ vanishes on the diagonal and is therefore a cusp form; so $\mathrm{Res}_{\lambda_5} f$ is also a cusp form (and is therefore divisible by $s_4$) and has weight $2$ mod $4$ (and is therefore divisible by $e_2$). The remainder $\frac{\mathrm{Res}_{\lambda_5} f}{e_2 s_4}$ has eigenvalue $+1$ under the involution $W_5$ and is therefore a polynomial in $e_2^2, s_4$. Since $$e_2 s_4 = \mathrm{Res}_{\lambda_5} \phi_3$$ it follows from case 1 that $\mathrm{Res}_{\lambda_5} f$ coincides with the restriction of some polynomial in $\mathbf{E}_2,\phi_2,\phi_3$. \\

Case 3: suppose $f \in M_k^{sym}(\Gamma_K)$ is symmetric and has odd weight. In particular it has a forced zero on $T_{29}$. After multiplying by $\phi_2$, we obtain a form $\phi_2 f$ which vanishes on $T_{29}$ and vanishes to order at least three along the diagonal and is therefore divisible by $\phi_6$. The remainder $\frac{f \phi_2}{\phi_6}$ is antisymmetric of odd weight so by the argument in case $2$ we find $$\mathrm{Res}_{\lambda_5} (f \phi_2 \phi_6^{-1}) = e_2 s_4 P(e_2^2,s_4)$$ for some polynomial $P$, and therefore $$\mathrm{Res}_{\lambda_5} f = \frac{\mathrm{Res}_{\lambda_5} \phi_6}{\mathrm{Res}_{\lambda_5} \phi_2} \cdot e_2 s_4 P(e_2^2,s_4) = e_2 e_4 s_4^2 P(e_2^2,s_4).$$  Since $e_2 e_4 s_4^2 = \mathrm{Res}_{\lambda_5} \phi_7$ and $e_2^2, s_4$ were accounted for in case 1, we see that $\mathrm{Res}_{\lambda_5} f$ is the restriction of some polynomial in $\mathbf{E}_2,\phi_2,\phi_7$. \\

Case 4: suppose $f \in M_k^{anti}(\Gamma_K)$ is antisymmetric and has even weight; in this case there are forced zeros on both $T_1$ and $T_{29}$. After multiplying by $\phi_3$ we see that $\phi_3 f$ is divisible by $\phi_6$, and the remainder is antisymmetric of odd weight. By the argument in case $2$ again we find $$\mathrm{Res}_{\lambda_5} (f \phi_3 \phi_6^{-1}) = e_2 s_4 P(e_2^2,s_4)$$ for some polynomial $P$, and therefore $$\mathrm{Res}_{\lambda_5} f = \frac{\mathrm{Res}_{\lambda_5} \phi_6}{\mathrm{Res}_{\lambda_5} \phi_3} \cdot e_2 s_4 P(e_2^2,s_4) = e_4 s_4^2 P(e_2^2,s_4).$$ Since $e_4 s_4^2 = \mathrm{Res}_{\lambda_5} \phi_6$ we see that $\mathrm{Res}_{\lambda_5} f$ is the restriction of some polynomial in $\mathbf{E}_2,\phi_2,\phi_6$.
\end{proof}

\begin{prop} Every Hilbert modular form of parallel weight for $\mathbb{Q}(\sqrt{29})$ is an isobaric polynomial in the generators $$\mathbf{E}_2,\phi_2,\phi_3,\psi_3,\phi_4,\phi_5,\phi_6,\psi_6,\phi_7,\phi_8,\phi_9.$$
\end{prop}
\begin{proof} We use induction on the weight $k$ of $f$. If $k \le 0$, then $f$ is constant. \\

In general, we can assume by the previous lemma that $f$ vanishes on $T_5$. In particular, $f$ is a cusp form. As before, we split $f$ into its symmetric and antisymmetric parts and obtain four cases to distinguish. \\

Case 1: $k$ is odd and $f$ is symmetric. In particular $f$ has a forced zero on $T_{29}$. Then the restriction of $f$ to the diagonal is a cusp form of weight $2k \equiv 2 \, (4)$ and therefore has the form $$f(\tau,\tau) = \Delta(\tau) E_6(\tau) P(E_4,\Delta)$$ for some polynomial $P$; so $g := f - \phi_9 P(\mathbf{E}_2,\psi_6)$ vanishes on both $T_1$ and $T_{29}$. Actually $g$ must have at least a double zero along $T_1$ because $g \cdot \phi_4$ is divisible by $\phi_6$ and the quotient $\frac{g \phi_4}{\phi_6}$ is antisymmetric and therefore vanishes on the diagonal. Therefore, $g$ expands about the diagonal in the form $$g(\tau_1,\tau_2) = h(\tau_2) (\tau_1 - \tau_2)^2 + O((\tau_1 - \tau_2)^3),$$ where $h$ is a cusp form of weight $2k + 4 \equiv 2 \, (4)$ and therefore has the form $$h(\tau) = \Delta(\tau) E_6(\tau) P_2(E_4, \Delta)$$ for some polynomial $P_2$. On the other hand, the form $\phi_7$ has a double zero on $T_1$ and the expansion $$\phi_7(\tau_1,\tau_2) = \Delta(\tau_2) E_6(\tau_2) (\tau_1 - \tau_2)^2 + O((\tau_1 - \tau_2)^3)$$ up to a constant multiple (since $\Delta E_6$ is the unique normalized cusp form of weight $18$); for the argument here there is no loss of generality in assuming that constant is $1$. Therefore $$g - \phi_7 P_2(\mathbf{E}_2,\psi_6) = f - \phi_9 P(\mathbf{E}_2,\psi_6) - \phi_7 P_2(\mathbf{E}_2,\psi_6)$$ has a zero on $T_{29}$ and at least a triple zero on $T_1$ and is therefore divisible by $\phi_6$. The quotient $$\frac{f - \phi_9 P(\mathbf{E}_2,\psi_6) - \phi_7 P_2(\mathbf{E}_2,\psi_6)}{\phi_6}$$ has smaller weight than $f$, and the claim follows by induction. \\

Case 2: $k$ is odd and $f$ is antisymmetric. Expand $f$ about the diagonal in the form $$f(\tau_1,\tau_2) = h(\tau_2) (\tau_1 - \tau_2) + O((\tau_1 - \tau_2)^2);$$ then $h(\tau_2)$ is a cusp form of weight $2k+2 \equiv 0 \, (4)$ and therefore of the form $\Delta(\tau) P(E_4,\Delta)$ for some polynomial $P$. Since $\phi_5$ has the expansion $$\phi_5(\tau_1,\tau_2) = \Delta(\tau_2) (\tau_1 - \tau_2) + O((\tau_1 - \tau_2)^2)$$ (up to a possible constant multiple, $\Delta(\tau_2)$ is the unique cusp form of weight $12$), it follows that $$g = f - \phi_5 P(\mathbf{E}_2,\psi_6)$$ has at least a double zero along $T_1$ and still vanishes along $T_5$ (because $\phi_5$ does). Actually $g$ has at least a triple zero along $T_1$ because $g \phi_5$ is divisible by $\psi_3$ with the result being antisymmetric and therefore still vanishing along $T_1$. It follows that $g$ itself is divisble by $\psi_3$. The quotient $\frac{g}{\psi_3}$ has smaller weight than $f$, and the claim follows by induction. \\

Case 3: $k$ is even and $f$ is symmetric. The restriction $\mathrm{Res}_{\lambda_1} f$ has the form $\Delta(\tau) P(E_4,\Delta)$ for some polynomial $P$ and therefore $g = f - \psi_6 P(\mathbf{E}_2,\psi_6)$ vanishes along $T_1$ (and it still vanishes along $T_5$, because both $f$ and $\psi_6$ do). Actually it must have at least a double zero along $T_1$, because $g\phi_4$ is divisible by $\psi_3$ with the result being antisymmetric and therefore still vanishing along $T_1$. We expand $g$ along the diagonal as $$g(\tau_1,\tau_2) = h(\tau_2) (\tau_1 - \tau_2)^2 + O((\tau_1 - \tau_2)^3);$$ then $h$ is a cusp form of weight divisble by $4$ so it is $h = \Delta \cdot P_2(E_4,\Delta)$ for another polynomial $P_2$. With the expansion $$\phi_4(\tau_1,\tau_2) = \Delta(\tau_2) (\tau_1 - \tau_2)^2 + O((\tau_1 - \tau_2)^3)$$ (up to a constant multiple), it follows that $g - \phi_4 P_2(\mathbf{E}_2,\psi_6)$ has at least a triple zero along $T_1$ and continues to vanish along $T_5$ (since both $g$ and $\phi_4$ do). Therefore $g - \phi_4 P_2(\mathbf{E}_2,\psi_6)$ is divisible by $\psi_3$ with the quotient having smaller weight, and the claim follows by induction. \\

Case 4: $k$ is even and $f$ is antisymmetric (and therefore $f$ vanishes on both $T_1$ and $T_{29}$). Expand $f$ about the diagonal: $$f(\tau_1,\tau_2) = h(\tau_2) (\tau_1 - \tau_2) + O((\tau_1 - \tau_2)^2).$$ Then $h$ is a cusp form of weight $2$ mod $4$ and therefore $h = \Delta E_6 \cdot P(E_4,\Delta)$ for some polynomial $P$. Since $$\phi_8(\tau_1,\tau_2) = \Delta(\tau_2) E_6(\tau_2) (\tau_1 - \tau_2) + O((\tau_1 - \tau_2)^2)$$ (up to a constant multiple), it follows that $f - \phi_8 P(\mathbf{E}_2,\psi_6)$ has at least a double zero along $T_1$ (and continues to vanish on $T_5$, because $\phi_8$ does). As in the previous cases it follows that $f - \phi_8 P(\mathbf{E}_2,\psi_6)$ actually has at least a triple zero along $T_1$. It is therefore divisible by $\psi_3$ with the result having smaller weight, and the claim follows by induction.

\end{proof}

\begin{prop} The graded ring $M_*(\Gamma_{\mathbb{Q}(\sqrt{29})})$ is defined by the 11 generators $$\mathbf{E}_2,\phi_2,\phi_3,\psi_3,\phi_4,\phi_5,\phi_6,\psi_6,\phi_7,\phi_8,\phi_9$$ in weights $2,2,3,3,4,5,6,6,7,8,9$ and by 35 relations in weights 6 through 18.
\end{prop}
\begin{proof} The strategy we used to compute relations is as follows. There are totally positive elements $$\lambda_{23} = \frac{11 + \sqrt{29}}{2}, \; \; \lambda_{35} = \frac{13 + \sqrt{29}}{2}$$ of norms $23$ and $35$, respectively, which lie in the same Weyl chamber with respect to all $10$ of the Borcherds products considered above. Moreover the obstruction principle shows that there is a nearly-holomorphic modular form whose Fourier expansion begins $2q^{-35} + 2q^{-23} + 48 + ...$ and which lifts to a Borcherds product $\phi_{24}$ of weight $24$ with simple zeros along $T_{23}$ and $T_{35}$. Any nonzero Hilbert modular form $f$ whose restrictions to $T_{23}$ and $T_{35}$ both vanish is divisible by $\phi_{24}$; the quotient $\frac{f}{\phi_{24}}$ is also holomorphic in the cusps by the G\"otzky-Koecher principle and therefore has nonnegative weight so $f$ must have had weight at least $24$. \\

In this way we are able to compute all relations up to weight $24$ by computing relations only among the restrictions to $T_{23}$ and $T_{35}$ which is somewhat easier. The exact relations we found are tedious to write out and are left to the ancillary material.  We mention here that $\{\mathbf{E}_2,\phi_3,\psi_3\}$ is algebraically independent (because the restrictions of $\mathbf{E}_2$ and $\phi_3$ to $T_5$ are already algebraically independent, and $\psi_3$ vanishes there) and that any relation can be reduced against the relations we found up to weight $18$ to eliminate the variables other than $\mathbf{E}_2,\phi_3,\psi_3$. It follows that the 35 relations we found up to weight $18$ generate the entire ideal. The Hilbert series is \[ \sum_{k=0}^{\infty} \mathrm{dim}\, M_k(\Gamma_{\mathbb{Q}(\sqrt{29})}) t^k = \frac{(1 + t + t^2)(1 - t + t^2 + t^4 + t^6 - t^7 + t^8)}{(1 - t^2)(1 - t^3)^2}. \qedhere \]

%The exact relations we found are not particularly illuminating and are left to the supplementary material; we only mention here that they give expressions for $$\phi_3 \psi_3, \; \phi_3^2 \; (\text{weight} \, 6),$$ $$\phi_2 \phi_5, \; \phi_3 \phi_4, \; (\text{weight} \, 7),$$ $$\psi_4^2,\; \psi_3 \phi_5, \; \phi_3 \phi_5 \; (\text{weight} \, 8),$$ $$\psi_4 \phi_5, \; \psi_3 \psi_6, \; (\text{weight} \, 9),$$ $$\phi_5^2, \; \psi_4 \phi_6, \; \psi_3 \phi_7, \; \phi_3 \phi_7, \; (\text{weight} \, 10),$$ $$\phi_5 \phi_6, \; \psi_4 \phi_7, \; \psi_3 \phi_8, \; \phi_3 \phi_8, \; (\text{weight} \, 11),$$ $$\phi_6 \psi_6, \; \phi_6^2, \; \phi_5 \phi_7, \; \psi_4 \phi_8, \; \psi_3 \phi_9, \; (\text{weight} \, 12),$$ $$\psi_6 \phi_7, \; \phi_6 \phi_7, \; \phi_5 \phi_8 \; (\text{weight} \, 13),$$ $$\psi_6 \phi_8, \; \phi_6 \phi_8, \; \phi_7^2, \; (\text{weight} \, 14),$$ $$\phi_7 \phi_8, \; \phi_6 \phi_9 \; (\text{weight} \, 15),$$ $$\phi_8^2, \; \phi_7 \phi_9, \; (\text{weight} \, 16),$$  $$\phi_8 \phi_9, \; (\text{weight} \, 17),$$ $$\phi_9^2 \; (\text{weight} \, 18),$$ and therefore that in every parallel weight, the monomials in $\mathbf{E}_2,\phi_2,\phi_3,\psi_3,\psi_4,\phi_5,\phi_6,\psi_6,\phi_7,\phi_8,\phi_9$ that are not divisible by any of the above produce a basis.
\end{proof}

\section{The graded ring of Hilbert modular forms for $\mathbb{Q}(\sqrt{37})$}

Let $K = \mathbb{Q}(\sqrt{37})$ with ring of integers $\mathcal{O}_K = \mathbb{Z}[\omega]$, $\omega = \frac{1 + \sqrt{37}}{2}$. Since $2$ is inert and $3$ splits in $\mathcal{O}_K$, theorem 1 of \cite{BS} implies that the full character group of $SL_2(\mathcal{O}_K)$ is generated by two characters $\chi_1, \chi_{\omega}$ which are characterized by their effect on the translations $T_1,T_{\omega}$ by $1$ and $\omega$: $$\chi_1(T_1) = e^{2\pi i / 3}, \; \chi_1(T_{\omega}) = 1$$ and $$\chi_{\omega}(T_1) = 1, \; \chi_{\omega}(T_{\omega}) = e^{2\pi i / 3}.$$ The condition for a character $\chi$ being symmetric is $\chi(T_{\omega}) = \chi(T_{\omega'}) = \chi(T_1) \chi(T_{\omega}^{-1})$, i.e. $\chi(T_1) = \chi(T_{\omega})^2$. In particular the group of symmetric characters of $SL_2(\mathcal{O}_K)$ is cyclic of order three and generated by $\chi = \chi_1 \chi_{\omega}^{-1}$. \\

We compute input functions using the obstruction principle as in the last section; there is again a single cusp form that acts as the obstruction to extending arbitrary principal parts to nearly-holomorphic modular forms and it is $$q - q^3 - 2q^4 + 3q^7 - 2q^9 + 4q^{10} - 3q^{11} + 2q^{12} \pm ... \in S_2(\Gamma_0(37),\chi).$$ Therefore there exist nearly-holomorphic modular forms $F_1,G_1,F_2,G_2,F_4,G_7$ of weight $0$ and level $\Gamma_0(37)$ and Nebentypus $\chi(d) = \left(\frac{37}{d}\right)$ whose Fourier expansions begin as follows:

\begin{align*} F_1(\tau) &= 2q^{-3} + 2q^{-1} + 2 + 2q + 4q^3 + 2q^7 +0q^9 - 4q^{10} + 2q^{11} +4q^{12} \pm ... \\ G_1(\tau) &= 2q^{-4} + 4q^{-1} + 2 - 2q + 2q^3 - 2q^4 + 6q^7 + 8q^9 + 2q^{10} + 2q^{11} - 4q^{12} \pm ... \\ F_2(\tau) &= 2q^{-10} + 2q^{-4} + 2q^{-3} - 2q^{-1} + 4 - 6q - 8q^3 + 8q^4 + 18q^7 - 4q^9 + 6q^{10} + 34q^{11} + 30q^{12} \pm ... \\ G_2(\tau) &= 2q^{-10} + 4q^{-4} + 4 - 10q - 10q^3 + 6q^4 + 22q^7 + 4q^9 + 12q^{10} + 34q^{11} + 22q^{12} \pm ... \\ F_4(\tau) &= q^{-37} + 2q^{-1} + 8 + 24q + 42q^3 + 80q^4 + 170q^7 + 300q^9 + 416q^{10} + 504q^{11} + 664q^{12} + ... \\ G_7(\tau) &= q^{-37} + 2q^{-12} + 2q^{-3} + 14 + 28q + 60q^3 + 68q^4 + 150q^7 + 308q^9 + 450q^{10} + 520q^{11} + 700q^{12} + ... \end{align*}

The Borcherds lifts $\phi_i = \Psi_{F_i}$, $\psi_i = \Psi_{G_i}$ have weight $i$ and their divisors can be read off the principal parts of the input functions: $$\mathrm{div} \, \phi_1 = T_1 + T_3, \; \; \mathrm{div} \, \psi_1 = 2T_1 + T_4, \; \mathrm{div} \, \phi_2 = -T_1 + T_3 + T_4 + T_{10},$$ $$\mathrm{div} \, \psi_2 = 2T_4 + T_{10}, \; \mathrm{div} \, \phi_4 = T_1 + T_{37}, \; \mathrm{div} \, \psi_7 = T_3 + T_{12} + T_{37}.$$ In particular $\phi_1,\psi_1,\psi_2,\phi_4$ vanish on the diagonal and $\phi_2,\psi_7$ do not. Moreover $\phi_1,\psi_1,\phi_4$ are antisymmetric and $\phi_2,\psi_2,\psi_7$ are symmetric, and all of the products above have a nontrivial character of order $3$; the notation is chosen such that $\phi_1,\phi_2,\phi_4$ have the same character $\chi$ and that $\psi_1,\psi_2,\psi_7$ have the same character $\chi^{-1}$. All of this can be proved directly, but it is easier to read it off of the restrictions to the diagonal and to the curve $T_3$ which are worked out in the next paragraph. Note that by construction $G_2 = F_2 + G_1 - F_1$ and therefore $\phi_1 \psi_2 = \psi_1 \phi_2$. \\

We fix the totally positive integers $\lambda_1 = 1, \; \lambda_3 = \frac{7 + \sqrt{37}}{2}$. Then one can compute the restrictions $$\mathrm{Res}_{\lambda_1} \phi_2(\tau) = \eta(\tau)^8, \; \; \mathrm{Res}_{\lambda_1} \psi_7(\tau) = E_6(\tau) \eta(\tau)^{16}, \; \; \mathrm{Res}_{\lambda_1} \mathbf{E}_2(\tau) = E_4(\tau)$$ and \begin{align*} \mathrm{Res}_{\lambda_3} \psi_1(\tau) &= \eta(\tau)^2 \eta(3\tau)^2 = q^{1/3} (1 - 2q - q^2 + 5q^4 + 4q^5 - 7q^6 \pm ...) \\ \mathrm{Res}_{\lambda_3} \psi_2(\tau) &= \eta(\tau)^2 \eta(3\tau)^2 e_2(\tau) = q^{1/3} (1 + 10q + 11q^2 - 72q^3 + 29q^4 - 44q^5 + 29q^6 \pm ...) \\ \mathrm{Res}_{\lambda_3} \phi_4(\tau) &= \eta(\tau)^4 \eta(3 \tau)^4 e_4(\tau) = q^{2/3} (1 - 34q - 148q^2 + 454q^3 - 559q^4 + 2418q^5 + 680q^6 \pm ...) \\ \mathrm{Res}_{\lambda_3} \mathbf{E}_2(\tau) &= e_2(\tau)^2 = 1 + 24q + 216q^2 + 888q^3 + 1752q^4 + 3024q^5 + 7992q^6 \pm ...\end{align*} with the omitted restrictions above being zero. Here we use the notation $$e_2(\tau) = \frac{3 E_2(3\tau) - E_2(\tau)}{2} = 1 + 12q + 36q^2 + ..., \; \; e_4(\tau) = \frac{9E_4(3\tau) - E_4(\tau)}{8} = 1 - 30q - 270q^2 -...$$ Using these computations we can prove:

\begin{lem} Let $f \in M_*(\Gamma_K)$ be a Hilbert modular form of parallel weight and trivial character. Then there is a polynomial $P$ such that $f - P(\psi_1,\psi_2,\phi_4,\mathbf{E}_2)$ vanishes on the Hirzebruch-Zagier divisor $T_3$. If $f$ is symmetric or antisymmetric, then one can also choose $P(\psi_1,\psi_2,\phi_4,\mathbf{E}_2)$ to be symmetric or antisymmetric, respectively.
\end{lem}
\begin{proof} The ring of modular forms of level $\Gamma_0(3)$ and trivial character is generated by the forms \begin{align*} e_2(\tau) &= \frac{3 E_2(3\tau) - E_2(\tau)}{2} = 1 + 12q + 36q^2 + 12q^3 + ... \\ e_4(\tau) &= \frac{9 E_4(3\tau) - E_4(\tau)}{8} = 1 - 30q - 270q^2 - 570q^3 - ... \\ s_6(\tau) &= \eta(\tau)^6 \eta(3\tau)^6 = q - 6q^2 + 9q^3 + 4q^4 \pm ... \end{align*} where $e_2$ has weight $2$, $e_4$ has weight $4$, $s_6$ is a cusp form of weight $6$, together with a single relation $$e_4^2 = e_2^4 - 108 e_2 s_6$$ in weight $8$. All three generators are eigenforms of the Atkin-Lehner operator $W_3$ with eigenvalue $-1$. We split $f$ into its symmetric and antisymmetric parts and argue for each part separately. \\

Case 1: suppose $f \in M_k^{sym}(\Gamma_K)$ is symmetric and has even weight. Then $\mathrm{Res}_{\lambda_3}f$ is an eigenform of $W_3$ with eigenvalue $+1$ and has weight divisible by $4$; so it is a polynomial expression in $e_2^2, e_2 s_6, s_6^2$. Since $$e_2^2 = \mathrm{Res}_{\lambda_3} \mathbf{E}_2, \; e_2 s_6 = \mathrm{Res}_{\lambda_3}(\psi_1^2 \psi_2), \; s_6^2 = \mathrm{Res}_{\lambda_3} \psi_1^6,$$ we can find a polynomial in $\psi_1^6,\psi_1^2\psi_2,\mathbf{E}_2$ whose restriction to $T_3$ equals $\mathrm{Res}_{\lambda_3} f$. This will also be symmetric because $\psi_1^2, \psi_2, \mathbf{E}_2$ are symmetric. \\

Case 2: suppose $f \in M_k^{sym}(\Gamma_K)$ is symmetric and has odd weight. Then $f$ is a cusp form and therefore $\mathrm{Res}_{\lambda_3} f$ is also a cusp form; in particular it is divisble by $s_6$, and $s_6^{-1} \mathrm{Res}_{\lambda_3}$ has eigenvalue $-1$ under $W_3$. Therefore all monomials in $e_2,e_4,s_6$ that occur in $s_6^{-1} \mathrm{Res}_{\lambda_3} f$ must contain $e_4$ (as all expressions involving only $e_2,s_6$ in weights $0$ mod $4$ will have eigenvalue $+1$) so $\mathrm{Res}_{\lambda_3} f$ is $e_4 s_6$ multiplied by some polynomial in $e_2^2, e_2 s_6, s_6^2$. The claim follows from the previous case together with $$e_4 s_6 = \mathrm{Res}_{\lambda_3}(\phi_4 \psi_1),$$ and the resulting polynomial expression will again be symmetric because $\phi_4 \psi_1$ is symmetric. \\

Case 3: suppose $f \in M_k^{anti}(\Gamma_K)$ is antisymmetric. In particular $f$ vanishes along the diagonal and therefore also at the cusps; so $\mathrm{Res}_{\lambda_3} f$ is a cusp form that is an eigenform of $W_3$ with eigenvalue $-1$. After dividing by $s_6$, we reduce to case $1$; since $s_6 = \mathrm{Res}_{\lambda_3} \psi_1^3$ and $\psi_1^3$ is antisymmetric, the resulting polynomial expression is antisymmetric.
\end{proof}

This quickly implies the same result for all of the symmetric characters:

\begin{lem} Let $f \in M_*(\Gamma_K,\chi)$ be a Hilbert modular form of any weight with symmetric character $\chi$. Then there is a polynomial $P$ as above for which $f - P(\psi_1,\psi_2,\phi_4,\mathbf{E}_2)$ vanishes on the divisor $T_3$.
\end{lem}
\begin{proof} The nontrivial symmetric characters restrict on $T_3$ to the characters of $s_2(\tau) = \eta(\tau)^2 \eta(3\tau)^2$ and its square. Note that any modular form of level $\Gamma_0(3)$ that transforms under one of those characters is a cusp form (because the order at all cusps is nonintegral and in particular nonzero) and is therefore divisible by $s_2$ or $s_2^2$, respectively. Since $\mathrm{Res}_{\lambda_3} \psi_1 = s_2$, we can use the previous argument.
\end{proof}

By further restricting to the diagonal we can compute generators:

\begin{prop} Every Hilbert modular form $f \in M_*(\Gamma_K,\chi)$ for a symmetric character $\chi$ is an isobaric polynomial in $\phi_1,\psi_1,\phi_2,\psi_2,\phi_4,\psi_7,\mathbf{E}_2$.
\end{prop}
\begin{proof} We induct on the weight of $f$. If $f$ has weight zero, then it is constant. Otherwise, we again split $f$ into its symmetric and antisymmetric parts and argue for each part separately. \\

Case 1: suppose $f$ is symmetric and has odd weight. Then $f$ has a forced zero on $T_{37}$ so its restriction $g = \mathrm{Res}_{\lambda_1} f$ to the diagonal is a cusp form for $SL_2(\mathbb{Z})$ of weight $2$ mod $4$ for some cubic character, and is therefore some polynomial expression $P$ in $E_4, \eta^8, E_6 \eta^{16}$. More precisely, if $g$ has the character of $\eta^8$ then $\eta^{-8}g$ has weight $0$ mod $4$ and no character, and is therefore a polynomial in $E_4$ and $\Delta = (\eta^8)^3$. If $g$ has the character of $\eta^{16}$ then $\eta^{-16} g$ has weight $2$ mod $4$ and no character, and therefore is also divisible by $E_6$, and $E_6^{-1} \eta^{-16} g$ is holomorphic of weight $0$ mod $4$ so the previous sentence applies. Finally, if $g$ has the trivial character then it vanishes at $i$ and is therefore divisible by $E_6$. As $g$ is a cusp form, we can divide by $\Delta E_6 = \eta^8 \cdot E_6 \eta^{16}$ with the remainder having weight $0$ mod $4$ again. Altogether, we can write $$\mathrm{Res}_{\lambda_1} f = P(E_4,\eta^8,E_6 \eta^{16}) = \mathrm{Res}_{\lambda_1} P(\mathbf{E}_2,\phi_2,\psi_7).$$ In particular, $f - P(\mathbf{E}_2,\phi_2,\psi_7)$ vanishes on $T_1$ and is still symmetric and of odd weight so it continues to vanish on $T_{37}$. Therefore we can divide by $\phi_4$ (which has simple zeros only on $T_1$ and $T_{37}$) to obtain a holomorphic form $\frac{f - P(\mathbf{E}_2,\phi_2,\psi_7)}{\phi_4}$ of lower weight which is a polynomial as in the claim by induction. This implies the claim for $f$ itself. \\

Case 2: suppose $f$ is antisymmetric and has odd weight. We use the previous lemmas and assume without loss of generality that $f$ already vanishes on $T_3$. As an antisymmetric form, $f$ has a forced zero on $T_1$. In particular it is divisible by $\phi_1$ and $\frac{f}{\phi_1}$ has smaller weight than $f$ so the claim follows by induction. \\

Case 3: Now suppose that $f$ has even weight; and again by the previous lemmas, assume without loss of generality that $f$ already vanishes on $T_3$. Then $f$ is a cusp form, and its restriction to the diagonal is a cusp form of level $1$ whose weight is divisible by $4$ and which is therefore some polynomial in $E_4, \eta^8$ in which all monomials contain $\eta^8$. In particular, $\mathrm{Res}_{\lambda_1} f$ is the restriction of some polynomial $Q(\mathbf{E}_2,\phi_2)$ in which all monomials contain $\phi_2$. Since $\phi_2$ vanishes on $T_3$, we find that $f - Q(\mathbf{E}_2,\phi_2)$ vanishes on both $T_1$ and $T_3$, so it is divisible by $\phi_1$; as $\frac{f - Q(\mathbf{E}_2,\phi_2)}{\phi_1}$ has weight strictly less than that of $f$, the claim follows by induction.
\end{proof}

\begin{prop} The graded ring $M_{*,sym}(\Gamma_K)$ of modular forms with symmetric characters is presented by seven generators $$\phi_1,\psi_1,\mathbf{E}_2,\phi_2,\psi_2,\phi_4,\psi_7$$ of weights $1,1,2,2,2,4,7$ and by 9 relations $R_{3,1},R_{4,1},R_{4,\chi},R_{8,1},R_{8,\chi},R_{8,\chi^2},R_{9,\chi},R_{11,1},R_{14,\chi}$ in weights $3$ through $14$.
\end{prop}
\begin{proof} We take the positive-definite integers $$\lambda_{21} = \frac{11 + \sqrt{37}}{2}, \; \lambda_{33} = \frac{13 + \sqrt{37}}{2}$$ of norm $21$ and $33$ which lie in a common Weyl chamber for each of the Borcherds products above. The obstruction principle shows that there is a nearly-holomorphic modular form in $B_{37}$ whose Fourier expansion begins $2q^{-33} + 2q^{-21} + 32 + ...$ and which therefore lifts to a Hilbert modular form $f_{16}$ of weight $16$ with simple zeros exactly along $T_{21}$ and $T_{33}$. By the G\"otzky-Koecher principle (arguing as in the previous section) every nonzero modular form which vanishes along both $T_{21}$ and $T_{33}$ must have weight at least $16$ so we are able to find all relations up to weight $16$ by only computing relations among the restrictions to $T_{21},T_{33}$. \\

In this way we compute the relations
\begin{align*} R_{3,1} &= \psi_1 \phi_2 - \phi_1 \psi_2; \\ R_{4,1} &= 5\phi_2 \psi_2 - 5\phi_1 \psi_1 \mathbf{E}_2 - 4 \phi_1^2 \psi_1^2 - 20\phi_1^2 \phi_2; \\ R_{4,\chi} &= 5 \psi_2^2 - 5 \psi_1^2 \mathbf{E}_2 -4 \phi_1 \psi_1^3 - 20 \phi_1^2 \psi_2; \\ R_{8,1} &= 5 \phi_2^2 \phi_4 - 5 \phi_1^2 \mathbf{E}_2 \phi_4 - 20\phi_1 \psi_7 + 36 \phi_1^3 \psi_1 \phi_4; \\ R_{8,\chi} &= \psi_1 \psi_7 - \phi_1^2 \phi_2 \phi_4 - 2\phi_1^2 \psi_1^2 \phi_4; \\ R_{8,\chi^2} &= 125 \phi_4^2 - 125 \psi_1^4 \mathbf{E}_2^2 + 13500 \psi_1^6 \psi_2 - 250\phi_1 \psi_1 \psi_2 \mathbf{E}_2^2 + 38050\phi_1 \psi_1^5 \mathbf{E}_2 - 125 \phi_1^2 \mathbf{E}_2^3 + 27000\phi_1^2 \phi_2^3 \\ &\quad \quad + 35850\phi_1^2 \psi_1^2 \psi_2 \mathbf{E}_2 + 33895\phi_1^2 \psi_1^6 + 11700\phi_1^3 \psi_1 \mathbf{E}_2^2 + 163340 \phi_1^3 \psi_1^3 \psi_2 + 42750 \phi_1^4 \phi_2 \mathbf{E}_2 \\ &\quad\quad + 119710\phi_1^4 \psi_1^2 \mathbf{E}_2 + 82466 \phi_1^5 \psi_1^3 + 447700 \phi_1^6 \psi_2 + 3375 \phi_1^8; \\ R_{9,\chi} &= 5 \psi_2 \psi_7 - 5\psi_1^3 \mathbf{E}_2 \phi_4 - 20 \phi_1^2 \psi_7 - 10 \phi_1^2 \psi_1 \psi_2 \phi_4 + 36 \phi_1^4 \psi_1 \phi_4; \\ R_{11,1} &= 500\phi_4 \psi_7 -125\phi_1\phi_2^2 \mathbf{E}_2^3 + 27000\phi_1\phi_2^5 -1000\phi_1^2\psi_1^5 \mathbf{E}_2^2 + 108000\phi_1^2 \psi_1^7 \psi_2 + 125\phi_1^3 \mathbf{E}_2^4 \\ &\quad\quad + 15750\phi_1^3 \phi_2^3 \mathbf{E}_2 -2500\phi_1^3 \psi_1^2 \psi_2 \mathbf{E}_2^2 + 358400\phi_1^3 \psi_1^6 \mathbf{E}_2 -1900\phi_1^4 \psi_1 \mathbf{E}_2^3 + 439000\phi_1^4 \psi_1^3 \psi_2 \mathbf{E}_2 \\ &\quad\quad + 314360\phi_1^4 \psi_1^7 + 50\phi_1^5 \phi_2 \mathbf{E}_2^2 + 236200\phi_1^5 \psi_1^2 \mathbf{E}_2^2 + 1658300\phi_1^5 \psi_1^4 \psi_2 + 1725760\phi_1^6 \psi_1^3 \mathbf{E}_2 \\ &\quad\quad + 1554640\phi1^7 \psi_2 \mathbf{E}_2 + 2571775\phi_1^7 \phi_2^2 + 1182416\phi_1^7 \psi_1^4 + 6680424\phi_1^8 \psi_1 \psi_2 -3375\phi_1^9 \mathbf{E}_2 + 24300\phi_1^{10} \psi_1; \end{align*} \begin{align*} R_{14,\chi} &= 10000\psi_7^2 -625\phi_2^4 \mathbf{E}_2^3 + 135000\phi_2^7 + 1250\phi_1^2 \phi_2^2 \mathbf{E}_2^4 -56250\phi_1^2 \phi_2^5 \mathbf{E}_2 -625\phi_1^4 \mathbf{E}_2^5 -78500\phi_1^4 \phi_2^3 \mathbf{E}_2^2 \\ &\quad\quad -40000\phi_1^4 \psi_1^6 \mathbf{E}_2^2 + 4320000\phi_1^4 \psi_1^8 \psi_2 -120000\phi_1^5 \psi_1^3 \psi_2 \mathbf{E}_2^2 + 16496000\phi_1^5 \psi_1^7 \mathbf{E}_2 -56250\phi_1^6 \phi_2 \mathbf{E}_2^3 \\ &\quad\quad + 16746875\phi_1^6 \phi_2^4 -129600\phi_1^6 \psi_1^2 \mathbf{E}_2^3 + 24728000\phi_1^6 \psi_1^4 \psi_2 \mathbf{E}_2 + 14302400\phi_1^6 \psi_1^8 + 18188000\phi_1^7 \psi_1^3 \mathbf{E}_2^2 \\ &\quad\quad + 81259200\phi_1^7 \psi_1^5 \psi_2 + 6064800\phi_1^8 \psi_2 \mathbf{E}_2^2 + 24373050\phi_1^8 \phi_2^2 \mathbf{E}_2 + 109220400\phi_1^8 \psi_1^4 \mathbf{E}_2 \\ &\quad\quad + 133829600\phi_1^9 \psi_1 \psi_2 \mathbf{E}_2 + 73829440\phi_1^9 \psi_1^5 + 16875\phi_1^{10} \mathbf{E}_2^2 + 424151360\phi_1^{10} \psi_1^2 \psi_2 \\ &\quad \quad + 229181280\phi_1^{11} \psi_1 \mathbf{E}_2 + 917697120\phi_1^{12} \phi_2 + 184414224\phi_1^{12} \psi_1^2.\end{align*}

Any relation of weight greater than $16$ can be reduced against the relations above to eliminate the variables $\phi_2,\psi_2,\phi_4,\psi_7$, and the remaining generators $\phi_1,\psi_1,\mathbf{E}_2$ are algebraically independent (because the restrictions of $\psi_1, \mathbf{E}_2$ to $T_3$ are already algebraically independent, and $\phi_1$ vanishes along $T_3$). Therefore the relations above are enough. Let $\mathfrak{e}_1,\mathfrak{e}_{\chi},\mathfrak{e}_{\chi^2}$ be multiplicative symbols which represent the symmetric characters of $SL_2(\mathcal{O}_K)$; then we get the Hilbert series $$\sum_{k=0}^{\infty} \sum_{i \in \mathbb{Z}/3\mathbb{Z}} \mathrm{dim}\, M_{k,\chi^i}(\Gamma_K) t^k \mathfrak{e}_{\chi^i} = \frac{1 + t^2 (\mathfrak{e}_{\chi} + \mathfrak{e}_{\chi^2}) - t^3 + t^4 (\mathfrak{e}_{\chi} + \mathfrak{e}_{\chi^2}) - t^5 \mathfrak{e}_{\chi} + 2t^6 + t^6 \mathfrak{e}_{\chi^2} - t^7 \mathfrak{e}_{\chi}}{(1-t \mathfrak{e}_{\chi}) (1 - t \mathfrak{e}_{\chi^2})(1-t^2)}.$$ (This can be computed in Macaulay2, for example; the exponents and characters in the denominator come from $\phi_1,\psi_1,\mathbf{E}_2$.) With a bit of algebra (e.g. replacing $\frac{1}{1 - t \mathfrak{e}_{\chi}}$ by $\frac{1 +t \mathfrak{e}_{\chi} + t^2 \mathfrak{e}_{\chi^2}}{1 - t^3}$) it follows that the dimensions of Hilbert modular forms with trivial character have the generating series $$\sum_{k=0}^{\infty} \mathrm{dim}\, M_k(\Gamma_K) t^k = \frac{1 - t + t^2 + t^3 + t^4 + t^5 + t^6 - t^7 + t^8}{(1 - t)(1-t^2)(1-t^3)}.$$
\end{proof}

\begin{prop} The graded ring of modular forms $M_*(\Gamma_K)$ is presented by the $15$ generators $$\mathbf{E}_2, \phi_1 \psi_1, \phi_1^3, \phi_1 \psi_2, \psi_1^3, \phi_2 \psi_2, \psi_1^2 \psi_2, \psi_1 \phi_4, \phi_1 \phi_2^2, \phi_2^3, \psi_2 \phi_4, \phi_1^2 \phi_4, \phi_1 \phi_2 \phi_4, \phi_2^2 \phi_4, \phi_2 \psi_7$$ of weights $2,2,3,3,3,4,4,5,5,6,6,6,7,8,9$ and by 77 relations in weights 6 through 18.
\end{prop}
A number of the 77 defining relations are those that are implied by the notation, e.g. $(\phi_1 \psi_1)^3 = \phi_1^3 \cdot \psi_1^3$.
\begin{proof} As generators we can take the monomials in $\mathbf{E}_2,\phi_1,\psi_1,\phi_2,\psi_2,\phi_4,\psi_7$ that have trivial character (and that cannot be further split into monomials with trivial character). Many of these monomials turn out not to be necessary due to relations such as $\psi_1 \phi_2 = \phi_1 \psi_2$. The 15 given above are minimal in that sense. \\

The ideal of relations is given by intersecting the ideal of relations from the previous lemma with the subring of Hilbert modular forms with trivial character. Finding this intersection and minimal generators for it is a straightforward Gr\"obner basis computation which was done in Macaulay2. We leave the explicit relations to the supplementary material.
\end{proof}

\bibliographystyle{plainnat}
\bibliography{\jobname}

\end{document}